\documentclass[envcountsame,envcountsect]{svproc}
\usepackage{amsmath, verbatim, amsfonts, amssymb, color}
\usepackage{mathrsfs}
\usepackage{dsfont}
\usepackage{mathtools}
\usepackage[pdftex]{graphicx}
\renewcommand\labelenumi{\textup{(\arabic{enumi})}}
\renewcommand\theenumi\labelenumi
\numberwithin{equation}{section}

 \makeatletter
\def\@makefnmark{\hbox{(\@textsuperscript{\normalfont\@thefnmark})}}
\makeatother

\makeatletter
\@namedef{subjclassname@2020}{%
  \textup{2020} Mathematics Subject Classification}
\makeatother

\providecommand{\ack}[1]{\par\addvspace\baselineskip
\noindent\ackname\enspace\ignorespaces#1}%
\def\subjclassname{\textup{2020} \textit{Mathematics Subject Classification:}}%
\providecommand{\subjclass}[1]{\par\addvspace\baselineskip
\noindent\subjclassname\enspace\ignorespaces#1}%

\newcommand{\E}{{\mathbb{E}}}
\newcommand{\N}{{\mathbb{N}}}
\def\P{{\mathbb{P}}}
\newcommand{\R}{{\mathbb{R}}}
\newcommand{\Z}{{\mathbb{Z}}}
\newcommand{\cA}{{\mathcal{A}}}
\newcommand{\cB}{{\mathcal{B}}}
\newcommand{\cC}{{\mathcal{C}}}
\newcommand{\cE}{{\mathcal{E}}}
\newcommand{\cF}{{\mathcal{F}}}

\newcommand{\cK}{{\mathcal{K}}}
\newcommand{\cN}{{\mathcal{N}}}
\newcommand{\cP}{{\mathcal{P}}}
\newcommand{\cU}{{\mathcal{U}}}
\def\a{\alpha}
\def\b{\beta}
\newcommand{\gm}{\gamma}
\newcommand{\dl}{\delta}

\newcommand{\sg}{\sigma}
\newcommand{\zt}{\zeta}
\newcommand{\lm}{\lambda}
\newcommand{\ph}{\varphi}
\newcommand{\om}{\omega}
\newcommand{\Sg}{\Sigma}
\newcommand{\Om}{\Omega}
\newcommand{\SG}{\mathrm{SG}}
\newcommand{\la}{\langle}
\newcommand{\ra}{\rangle}

\newcommand{\bfone}{\mathbf{1}}
\newcommand{\bfnu}{\boldsymbol{\nu}}

\begin{document}
\mainmatter
\title{\bfseries Singularity of energy measures on a class of inhomogeneous Sierpinski gaskets}
\titlerunning{Singularity of energy measures}

\author{Masanori Hino\inst{1} \and Madoka Yasui\inst{2}}
\authorrunning{Masanori Hino and Madoka Yasui}
\tocauthor{Masanori Hino and Madoka Yasui}
\institute{Department of Mathematics, Kyoto University, Kyoto 606-8502, Japan\\
\email{hino@math.kyoto-u.ac.jp}\medskip\and
Katsushika-ku, Tokyo, Japan}

\maketitle

\begin{abstract}
We study energy measures of canonical Dirichlet forms on inhomogeneous Sierpinski gaskets. We prove that the energy measures and suitable reference measures are mutually singular under mild assumptions.

\ack{This study was supported by JSPS KAKENHI Grant Numbers JP19H00643 and JP19K21833.}
\end{abstract}

\bigskip\noindent
\section{Introduction}
Energy measures associated with regular Dirichlet forms are fundamental concepts in stochastic analysis and related fields. For example, the intrinsic metric is defined by using energy measures and appears in Gaussian estimates of the transition probabilities. Energy measures are also crucial for describing the conditions for sub-Gaussian behaviors of transition densities. The energy measures are expected to be singular with respect to (canonical) underlying  measures for canonical Dirichlet forms on self-similar fractals, which has been confirmed in many cases~\cite{Kus89,BST99,Hi05,HN06}. Recently, such a singularity was proved under full off-diagonal sub-Gaussian estimates of the transition densities~\cite{KM20}.

In this paper, we study a class of inhomogeneous Sierpinski gaskets as examples that have not yet been covered in the previous studies: they do not necessarily have strict self-similar structures or nice sub-Gaussian estimates. We show that the singularity of the energy measures still holds under mild assumptions. The strategy of our proof is based on quantitative estimates of probability measures on shift spaces, the techniques of which were used in \cite{Hi05,HN06}. We expect this study to lead to further progress in stochastic analysis of complicated spaces of this kind.

This paper is organized as follows: In Section~2, we introduce a class of inhomogeneous Sierpinski gaskets and canonical Dirichlet forms defined on them, and state the main results. In Sections~3 and 4, we provide preliminary lemmas and prove the theorems. In Section~5, we make some concluding remarks.

\section{Framework and statement of theorems}
We begin by recalling $2$-dimensional level-$\nu$ Sierpinski gaskets $\SG(\nu)$ for $\nu\ge2$. 
Let $N(\nu)=\nu(\nu+1)/2$.
Let $\tilde K$ be an equilateral triangle in $\R^2$ including the interior. Let $K_i^{(\nu)}\subset \tilde K$, $i=1,2,\dots,N(\nu)$, be equilateral triangles including the interior that are obtained by dividing the sides of $\tilde K$ in $\nu$, joining these points, and removing all the downward-pointing triangles, as in Figure~\ref{fig:1}.
\begin{figure}[t]
\centering
\includegraphics[width=\hsize]{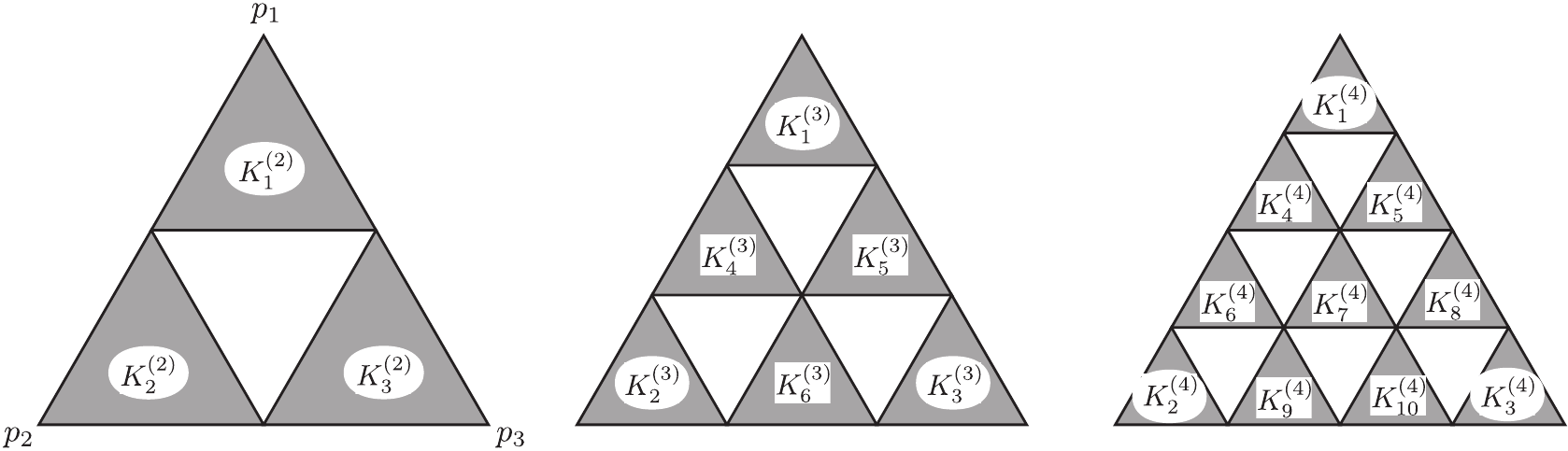}
\caption{$K_i^{(\nu)}$, the image of $\tilde K$ by the contractive affine map $\psi_i^{(\nu)}$ $(\nu=2,3,4)$.}
\label{fig:1}
\end{figure}
Let $\psi_i^{(\nu)}$, $i=1,2,\dots,N(\nu)$, be the contractive affine map from $\tilde K$ onto $K_i^{(\nu)}$ of type $\psi_i^{(\nu)}(z)=\nu^{-1}z+\a_i^{(\nu)}$ for some $\a_i^{(\nu)}\in \R^2$. 
Then, the $2$-dimensional level-$\nu$ Sierpinski gasket $\SG(\nu)$ is defined as a unique non-empty compact subset $K$ in $\tilde K$ such that
\[
K=\bigcup_{i=1}^{N(\nu)}\psi_i^{(\nu)}(K).
\]
Let $S_0=\{1,2,3\}$, and let $V_0=\{p_1,p_2,p_3\}$ be the set of all vertices of $\tilde K$.
In the definition of $\SG(\nu)$, the labeling of $K_i^{(\nu)}$ does not matter.
For later convenience, we assign $K_i^{(\nu)}$ for $i\in S_0$ to the triangle that contains $p_i$. As a result, $\psi_i^{(\nu)}$ has a fixed point $p_i$.

For a general non-empty set $X$, denote by $l(X)$ the set of all real-valued functions on $X$. When $X$ is finite, the inner product $(\cdot,\cdot)$ on $l(X)$ is defined by
\[
  (x,y)=\sum_{p\in X} x(p)y(p), \quad x,y\in l(X).
\]
We regard $l(X)$ as the $L^2$-space on $X$ equipped with the counting measure. Then, the $L^2$-inner product is identical with $(\cdot,\cdot)$. The induced norm is denoted by $|\mathrel{\cdot}|$.

A symmetric linear operator $D=(D_{p,q})_{p,q\in V_0}$ on $l(V_0)$ is defined as
\[
D_{p,q}=\begin{cases}
-2 & \text{if }p=q,\\
1 & \text{otherwise}.
\end{cases}
\]
Let 
\[
Q(x,y):=(-Dx,y)=-\sum_{p,q\in V_0}D_{p,q}x(q)y(p)\]
for $x,y\in l(V_0)$.
More explicitly,
\begin{align*}
Q(x,y)&=(x(p_1)-x(p_2))(y(p_1)-y(p_2))
+(x(p_2)-x(p_3))(y(p_2)-y(p_3))\\
&\quad +(x(p_3)-x(p_1))(y(p_3)-y(p_1)).
\end{align*}
This is a Dirichlet form on $l(V_0)$.
To simplify the notation, we sometimes write $Q(x)$ for $Q(x,x)$.

Let
\[
 V_1^{(\nu)}=\bigcup_{i=1}^{N(\nu)}\psi_i^{(\nu)}(V_0).
\]
Let $r^{(\nu)}>0$ and $Q^{(\nu)}$ be a symmetric bilinear form on $V_1^{(\nu)}$ that is defined by
\[
Q^{(\nu)}(x,y)=\sum_{i=1}^{N(\nu)}\frac1{r^{(\nu)}} Q(x\circ \psi_i^{(\nu)}|_{V_0},y\circ \psi_i^{(\nu)}|_{V_0}),\quad
x,y\in l(V_1^{(\nu)}).
\]
Then, there exists a unique $r^{(\nu)}>0$ such that, for every $x\in l(V_0)$,
\begin{equation}\label{eq:consistency}
Q(x,x)=\inf\bigl\{Q^{(\nu)}(y,y)\bigm| y\in l(V_1^{(\nu)})\text{ and }y|_{V_0}=x\bigr\}.
\end{equation}
Hereafter, we fix such $r^{(\nu)}$.
For example, $r^{(2)}=3/5$, $r^{(3)}=7/15$, and $r^{(4)}=41/103$, which are confirmed by the concrete calculation.

For each $x\in l(V_0)$, there exists a unique $y\in l(V_1)$ that attains the infimum in \eqref{eq:consistency}. For $i=1,2,\dots,N(\nu)$, the map $l(V_0)\ni x\mapsto y\circ \psi_i^{(\nu)}|_{V_0}\in l(V_0)$ is linear, which is denoted by $A_i^{(\nu)}$. Then, it holds that
\begin{equation}\label{eq:consistency2}
Q(x,x)=\sum_{i=1}^{N(\nu)}\frac1{r^{(\nu)}} Q(A_i^{(\nu)}x,A_i^{(\nu)}x),\quad x\in l(V_0).
\end{equation}
We can construct a Dirichlet form on $\SG(\nu)$ by using such data, but we omit the explanation because we discuss it in more general situations soon.

For reference, we give a quantitative estimate of $r^{(\nu)}$.
\begin{lemma}\label{lem:r}
$1/\nu<r^{(\nu)}<N(\nu)/\nu^2$.
\end{lemma}
\begin{proof}
This kind of inequality should be well-known (see, e.g., \cite[Theorem~1]{Ba04}), and see the proof of \cite[Proposition~5.3]{KM20} (and also \cite[Proposition~6.30]{Ba98}) for the second inequality. For the first inequality, let 
\begin{equation}\label{eq:alpha}
\a=\inf\{Q(z,z)\mid z\in l(V_0),\ z(p_1)=1,\ z(p_2)=0\}>0.
\end{equation}
Then, for general $z\in l(V_0)$,
\begin{equation}\label{eq:lower}
Q(z,z)\ge (z(p_1)-z(p_2))^2 \a
\end{equation}
by considering $(z-z(p_2))/(z(p_1)-z(p_2))$.

The infimum of \eqref{eq:alpha} is attained by $x\in l(V_0)$ given by $x(p_1)=1$, $x(p_2)=0$, $x(p_3)=1/2$ (and $\a=3/2$).
Take $y\in l(V_1^{(\nu)})$ attaining the infimum of \eqref{eq:consistency}.
Let $I\subset\{1,2,\dots,N(\nu)\}$ be a $\nu$-points set such that, for each $i\in I$, the intersection of $\psi_i^{(\nu)}(V_0)$ and the segment connecting $p_1$ and $p_2$ is a two-points set, say $\{\check p_i,\hat p_i\}$. Note that $3\notin I$, and $y$ is not constant on $\psi_3^{(\nu)}(V_0)$, which is confirmed by applying the maximum principle (see, e.g., \cite[Proposition~2.1.7]{Ki}) to the graph whose vertices are all points of $V_1^{(\nu)}$ included in the triangle with $p_1$, $p_3$ and the middle point of $p_1$ and $p_2$ as the three vertices. Therefore,
\begin{align*}
\a&=Q^{(\nu)}(y,y)\\
&>\sum_{i\in I}\frac1{r^{(\nu)}} Q(y\circ \psi_i^{(\nu)}|_{V_0},y\circ \psi_i^{(\nu)}|_{V_0})\\
&\ge\frac1{r^{(\nu)}}\sum_{i\in I}(y(\check p_i)-y(\hat p_i))^2\a\quad\text{(from \eqref{eq:lower})}\\
&\ge\frac{\a}{r^{(\nu)}}\Biggl(\sum_{i\in I}(y(\check p_i)-y(\hat p_i))\Biggr)^2\Biggl(\sum_{i\in I}1\Biggr)^{-1}\\
&=\frac{\a}{r^{(\nu)}}\cdot1\cdot\nu^{-1}.
\end{align*}
Thus, $1/\nu<r^{(\nu)}$.
\qed
\end{proof}
See also \cite{HK02} for the asymptotic behavior of $r^{(\nu)}$ as $\nu\to\infty$.

We now introduce $2$-dimensional inhomogeneous Sierpinski gaskets. 
We fix a non-empty finite subset $T$ of $\{\nu\in\N\mid \nu\ge 2\}$. For each $\nu\in T$, let $S^{(\nu)}$ denote the set of the letters $i^\nu$ for $i=1,2,\dots,N(\nu)$.
We set $S=\bigcup_{\nu\in T} S^{(\nu)}$ and $\Sg= S^\N$. For example, if $T=\{2,3\}$, then
\begin{align*}
S^{(2)}&=\{1^2,2^2,3^2\},\quad S^{(3)}=\{1^3,2^3,3^3,4^3,5^3,6^3\},
\end{align*}
and $S=S^{(2)}\cup S^{(3)}$ has nine elements.
(Note that $i^\nu$ \emph{does not mean} $\underbrace{ii\cdots i}_{\nu}$, the $\nu$-letter word consisting of only $i$, in this paper.)

For each $v\in S$, a shift operator $\sg_v\colon \Sg\to\Sg$ is defined by $\sg_v(\om_1\om_2\cdots)=v\om_1\om_2\cdots$.
Let $W_0=\{\emptyset\}$ and $W_m=S^m$ for $m\in\N$, and define $W_{\le n}=\bigcup_{m=0}^n W_m$ and $W_*=\bigcup_{m\in\Z_+} W_m$. 
Here, $\Z_+:=\N\cup\{0\}$.
For $w\in W_m$, $|w|$ represents $m$ and is called the length of $w$.
For $w=w_1\cdots w_m\in W_m$ and $w'=w'_1\cdots w'_n\in W_n$, $ww'\in W_{m+n}$ denotes $w_1\cdots w_m w'_1\cdots w'_n$. Also, $\sg_w\colon \Sg\to\Sg$ is defined as $\sg_w=\sg_{w_1}\circ\cdots\circ\sg_{w_m}$, and let $\Sg_w=\sg_w(\Sg)$.
For $k\le m$, $[w]_k$ denotes $w_1\cdots w_k\in W_k$. 
Similarly, for $\om=\om_1\om_2\cdots\in\Sg$ and $n\in\N$, let $[\om]_n$ denote $\om_1\cdots\om_n\in W_n$.
By convention, $\sg_\emptyset\colon\Sg\to\Sg$ is the identity map, $[w]_0:=\emptyset\in W_0$ for $w\in W_*$, and $[\om]_0:=\emptyset\in W_0$ for $\om\in\Sg$.

For $i^\nu\in S$, we define $\psi_{i^\nu}:=\psi_i^{(\nu)}$ and $A_{i^\nu}:=A_i^{(\nu)}$.
For $w=w_1w_2\cdots w_m\in W_m$, $\psi_w$ denotes $\psi_{w_1}\circ\psi_{w_2}\circ\dots\circ\psi_{w_m}$ and $A_w$ denotes $A_{w_m}\cdots A_{w_2}A_{w_1}$. Here, $\psi_\emptyset$ and $A_\emptyset$ are the identity maps by definition.
For $\om\in \Sg$, $\bigcap_{m\in\Z_+}\psi_{[\om]_m}(\tilde K)$ is a one-point set $\{p\}$. The map $\Sg\ni \om\mapsto p\in \tilde K$ is denoted by $\pi$. The relation $\psi_{v}\circ\pi =\pi\circ \sg_v$ holds for $v\in S$.

Now, we fix $L=\{L_w\}_{w\in W_*}\in T^{W_*}$. That is, we assign each $w\in W_*$ to $L_w\in T$.
We set $\tilde W_0=\{\emptyset\}$ and
\[
\tilde W_m=\bigcup_{w\in\tilde W_{m-1}}\bigl\{w v\bigm|v\in S^{(L_w)}\bigr\}
\]
for $m\in\N$, inductively.
Define $\tilde W_*=\bigcup_{m\in\Z_+} \tilde W_m\subset W_*$, $\tilde\Sg=\{\om\in\Sg\mid [\om]_m\in\tilde W_m\text{ for all }m\in\Z_+\}$ and $G(L)=\pi(\tilde\Sg)$. It holds that
\[
G(L)=\bigcap_{m\in\Z_+}\bigcup_{w\in\tilde W_m}\psi_w(\tilde K).
\]
We call $G(L)$ an inhomogeneous Sierpinski gasket generated by $L$. See Figure~\ref{fig:2} for a few examples.
We equip $G(L)$ with the relative topology of $\R^2$.
\begin{figure}[t]
\centering
\includegraphics[width=0.46\hsize]{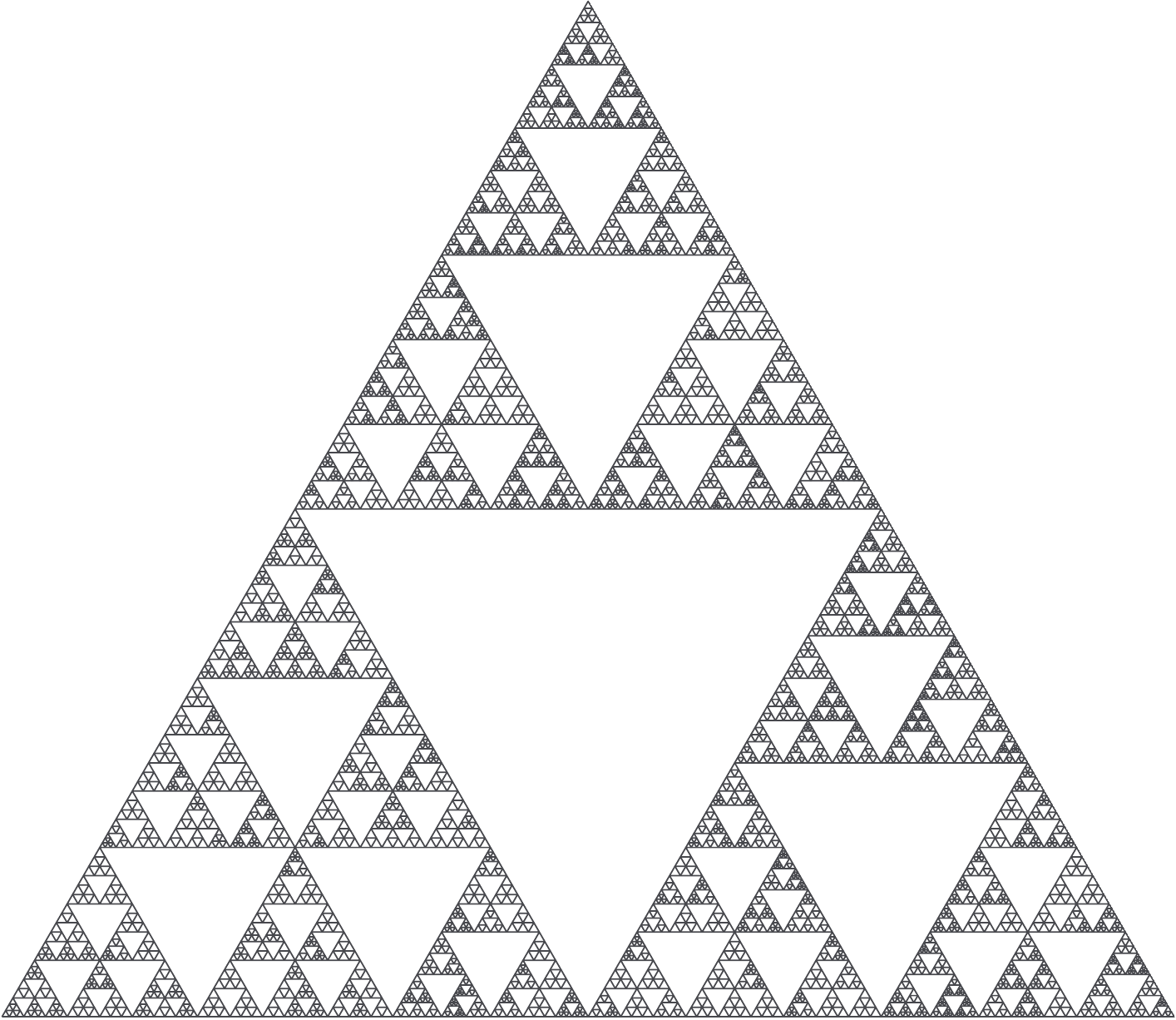}\qquad
\includegraphics[width=0.46\hsize]{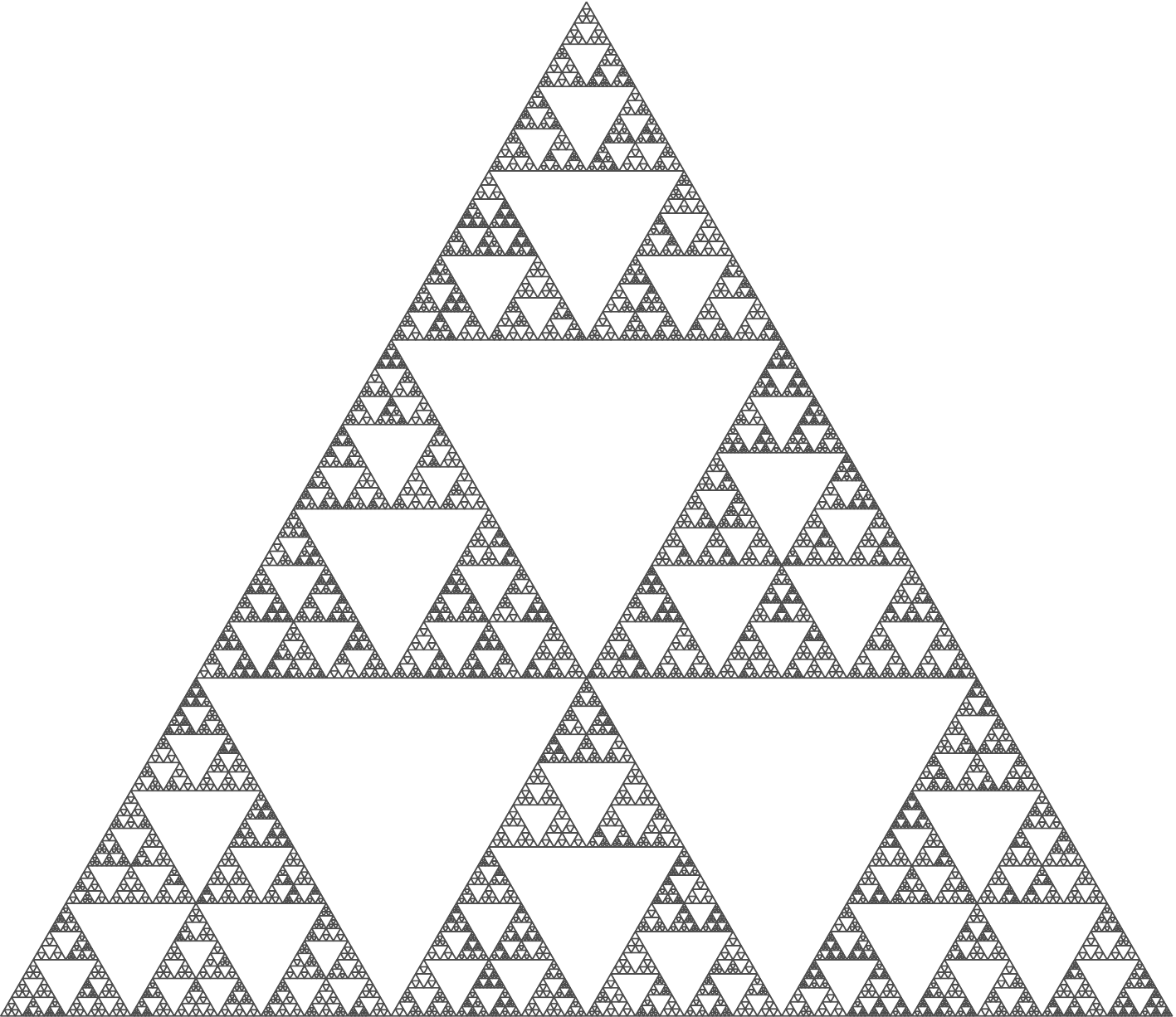}
\caption{Examples of inhomogeneous Sierpinski gaskets $(T=\{2,3\})$.}
\label{fig:2}
\end{figure}
If $L_w=\nu$ for all $w\in W_*$, then $G(L)$ is nothing but $\SG(\nu)$.

For $m\in\N$, let
\[
V_m=\bigcup_{w\in\tilde W_{m}}\psi_w(V_0),
\]
and let $V_*=\bigcup_{m\in\Z_+}V_m$.
The closure of $V_*$ is equal to $G(L)$.

Next, we define reference measures on $G(L)$.
Let
\[
\cA^{(\nu)}=\Biggl\{q=\{q_v\}_{v\in S^{(\nu)}}\Biggm|
q_v>0\text{ for all $v\in S^{(\nu)}$ and }\sum_{v\in S^{(\nu)}}q_v=1\Biggr\}
\]
and
\[
\cA=\Bigl\{q=\{q_v\}_{v\in S}\Bigm| \text{for each $\nu\in T$, }\{q_v\}_{v\in S^{(\nu)}}\in\cA^{(\nu)}\Bigr\}.
\]
For $q\in\cA$, there exists a unique Borel probability measure $\lm_q$ on $\Sg$ such that
\[
  \lm_q(\Sg_w)=\begin{cases}
  q_{w_1}\cdots q_{w_m} & \text{if }w=w_1\cdots w_m\in \tilde W_m,\\
  0 & \text{if }w\notin \tilde W_*.
  \end{cases}
\]
We note that
\[
\lm_q(\Sg\setminus \tilde\Sg)=\lim_{m\to\infty}\lm_q\Biggl(\Sg\setminus \bigcup_{w\in\tilde W_m}\Sg_w\Biggr)=0.
\]
In what follows, $q_w$ denotes $q_{w_1}\cdots q_{w_m}$ for $w=w_1\cdots w_m\in W_m$. By definition, $q_\emptyset=1$.
The Borel probability measure $\mu_q$ on $G(L)$ is defined by $\mu_q=(\pi|_{\tilde\Sg})_* \lm_q$, that is, the image measure of $\lm_q$ by $\pi|_{\tilde\Sg}\colon\tilde\Sg\to G(L)$. It is easy to see that $\mu_q$ has full support and does not charge any one points.
When $T=\{\nu\}$, $\mu_q$ is a self-similar measure on $G(L)=\SG(\nu)$.

We next construct a Dirichlet form on $G(L)$. 
Let $r_{i^\nu}=r^{(\nu)}$ for $i^\nu\in S$, and $r_w=r_{w_1}\cdots r_{w_m}$ for $w=w_1\cdots w_m\in W_m$. By definition, $r_\emptyset=1$.
For $m\in\Z_+$, let
\[
\cE^{(m)}(x,y)=\sum_{w\in\tilde W_m}\frac1{r_w} Q(x\circ\psi_w|_{V_0},y\circ\psi_w|_{V_0}),\quad
x,y\in l(V_m).
\]
From \eqref{eq:consistency} and \eqref{eq:consistency2}, it holds that for every $m\in\Z_+$ and $x\in l(V_m)$,
\[
\cE^{(m)}(x,x)=\inf\{\cE^{(m+1)}(y,y)\mid y\in l(V_{m+1})\text{ and }y|_{V_m}=x\}.
\]
Thus, for any $x\in l(V_*)$, the sequence $\{\cE^{(m)}(x|_{V_m},x|_{V_m})\}_{m=0}^\infty$ is non-decreasing.
We define
\begin{align*}
\cF&=\left\{f\in C(G(L))\;\middle|\;\lim_{m\to\infty}\cE^{(m)}(f|_{V_m},f|_{V_m})<\infty\right\},\\
\cE(f,g)&=\lim_{m\to\infty}\cE^{(m)}(f|_{V_m},g|_{V_m}),\quad f,g\in\cF,
\end{align*}
where $C(G(L))$ denotes the set of all real-valued continuous functions on $G(L)$.
Then, $(\cE,\cF)$ is a resistance form and also a strongly local regular Dirichlet form on $L^2(G(L),\mu_q)$ for any $q\in\cA$ (see \cite{Ha97} and \cite[Chapter~2]{Ki}). Here, $C(G(L))$ is regarded as a subspace of $L^2(G(L),\mu_q)$. We equip $\cF$ with the inner product $(f,g)_\cF:=\cE(f,g)+\int_{G(L)}fg\,d\mu_q$ as usual.

The energy measure $\mu_{\la f\ra}$ of $f\in \cF$ is a finite Borel measure on $G(L)$, which is characterized by
\[
\int_{G(L)} g\,d\mu_{\la f\ra}=2\cE(f,fg)-\cE(f^2,g),
\quad g\in\cF.
\]
By letting $g\equiv1$, the total mass of $\mu_{\la f\ra}$ is $2\cE(f,f)$.
Another expression of $\mu_{\la f\ra}$ is discussed in Section~3.

We introduce the following conditions for $q=\{q_v\}_{v\in S}\in\cA$ to describe our main theorem.
\begin{enumerate}
\item[\textup(A)] $q_{i^\nu}\ne r^{(\nu)}$ for all $i\in S_0$ and $\nu\in T$.
\item[\textup(B)] For each $l_0,l_1\in\N$, there exists $l_2\in\N$ such that the following~\textup{($\star$)} holds for $\mu_q$-a.e.\,$\om\in \Sg$:
\begin{itemize}
\item[\textup{($\star$)}] there exist infinitely many $k\in\Z_+$ such that, for every $i,j\in S_0$, 
\begin{align}
&[\om]_k i^{\nu_{k+1}}\cdots i^{\nu_{k+l_0}}j^{\nu_{k+l_0+1}}\cdots j^{\nu_{k+l_0+l_1}}j^{\nu_{k+l_0+l_1+1}}\cdots j^{\nu_{k+l_0+l_1+l_2}}\nonumber\\
&\in \tilde W_{k+l_0+l_1+l_2}
\label{eq:A1}
\end{align}
implies that
\begin{align}
&\{\nu_{m}\in T\mid k+l_0+1\le m\le k+l_0+l_1\}\nonumber\\
&\subset \{\nu_{m}\in T\mid k+l_0+l_1+1\le m\le k+l_0+l_1+l_2\}.\label{eq:A2}
\end{align}
\end{itemize}
\end{enumerate}
\begin{remark}
\begin{enumerate}
\item Condition~($\star$) is meaningful only for $\om\in\tilde\Sg$.
\item For $\om\in\tilde\Sg$, $k\in\Z_+$, and $i,j\in S_0$, the elements $\nu_{k+1},\nu_{k+2},\dots,\nu_{k+l_0+l_1+l_2}\in T$ so that \eqref{eq:A1} holds are uniquely determined.
Indeed, $\nu_{k+1}=L_{[\om]_k}$, $\nu_{k+2}=L_{[\om]_k i^{\nu_{k+1}}}$, $\nu_{k+3}=L_{[\om]_k i^{\nu_{k+1}}i^{\nu_{k+2}}}$, and so on.
\item A simple sufficient condition for \eqref{eq:A2} is
\begin{equation}\label{eq:A2'}
\{\nu_{m}\mid k+l_0+l_1+1\le m\le k+l_0+l_1+l_2\}=T.
\end{equation}
\end{enumerate}
\end{remark}
\begin{theorem}\label{th:main}
Let $q\in\cA$. Suppose that Condition~\textup{(A)} or \textup{(B)} holds.
Then, $\mu_{\la f\ra}$ and $\mu_q$ are mutually singular for every $f\in\cF$.
\end{theorem}
We provide some typical examples.
\begin{example}
Let $\nu\in T$ and define $L=\{L_w\}_{w\in W_*}$ by $L_w=\nu$ for all $w\in W_*$. Then, $G(L)$ is equal to $\SG(\nu)$. 
In this case, Condition~($\star$) is trivially satisfied for all $\om\in\tilde\Sg$ by letting $l_2=1$ because both sides of \eqref{eq:A2} are equal to $\{\nu\}$. Thus, by Theorem~\ref{th:main}, $\mu_{\la f\ra}\perp\mu_q$ for every $f\in\cF$ and $q\in\cA$.
This singularity has been proved in \cite{HN06} already.
\end{example}
\begin{example}
Take any sequence $\{\tau_m\}_{m\in\Z_+}\in T^{\Z_+}$ and let $L_w=\tau_{|w|}$ for $w\in W_*$. The set $G(L)$ associated with $L=\{L_w\}_{w\in W_*}$ has been studied in, e.g., \cite{Ha92,BH97,KM20}, and called a scale irregular Sierpinski gasket. 
\begin{enumerate}
\item Let $q=\{q_w\}_{w\in S}\in\cA$ be given by $q_v=N(\nu)^{-1}$ for $v\in S^{(\nu)}$. The associated measure $\mu_q$ is regarded as a uniform measure on $G(L)$. Since $N(\nu)^{-1}<\nu^{-1}$, Condition~(A) holds from Lemma~\ref{lem:r}. Therefore, $\mu_{\la f\ra}\perp \mu_q$ for any $f\in \cF$ from Theorem~\ref{th:main}. This case was discussed in \cite[Section~5]{KM20}.
\item \begin{enumerate}
\item Suppose that there exists $l_2\in\N$ such that $\{\tau_{k+1},\tau_{k+2},\dots,\tau_{k+l_2}\}=T$ for infinitely many $k\in\Z_+$. Then, Condition~($\star$) is satisfied for all $\om\in\tilde\Sg$, in view of \eqref{eq:A2'}. 
\item Suppose that for each $l\in\N$ there exists $k\in\Z_+$ such that $\tau_{k+1}=\tau_{k+2}=\dots=\tau_{k+l}$. Then, Condition~($\star$) with $l_2=1$ is satisfied for all $\om\in\tilde\Sg$ and $l_0,l_1\in\N$, but \eqref{eq:A2'} may fail to hold for any $l_2$. 
\end{enumerate}
In either case, $\mu_{\la f\ra}\perp\mu_q$ for any $f\in \cF$ and any $q\in\cA$ from Theorem~\ref{th:main}.
\end{enumerate}
\end{example}
\begin{example}
Let $\rho$ be a probability measure on $T$ with full support. We take a family of $T$-valued i.i.d.\ random variables $\{L_w(\cdot)\}_{w\in W_*}$ with distribution $\rho$ that are defined on some probability space $(\hat\Omega,\hat\cB,\hat P)$. For each $\hat\om\in\hat\Omega$, we can define an inhomogeneous Sierpinski gasket $G(L(\hat\om))$ associated with $L(\hat\om):=\{L_w(\hat\om)\}_{w\in W_*}$. This is called a random recursive Sierpinski gasket~\cite{Ha97}. 
Then, the following holds.
\begin{theorem}\label{th:RSG}
For $\hat P$-a.s.\ $\hat\om$, $G(L(\hat\om))$ satisfies Condition~\textup{(B)} for all $q\in\cA$. That is, for $\hat P$-a.s.\ $\hat\om$, the Dirichlet form on $G(L(\hat\om))$ can apply Theorem~\ref{th:main} for all $q\in\cA$ to conclude that the energy measures and $\mu_q$ are mutually singular for all $q\in\cA$.
\end{theorem}
\end{example}
Theorems~\ref{th:main} and \ref{th:RSG} are proved in Section~4.
\section{Preliminary lemmas}
In this section, we provide the necessary concepts and lemmas for proving Theorem~\ref{th:main}.
We fix $L=\{L_w\}_{w\in W_*}\in T^{W_*}$ and $q\in\cA$ and retain the notation used in the previous section.

For $w\in \tilde W_*$, let $K_w$ denote $\pi(\Sg_w\cap\tilde\Sg)\,({}=\psi_w(\tilde K)\cap G(L))$.

Let $m\in\Z_+$ and $x\in l(V_m)$. There exists a unique $h\in\cF$ that attains 
\[
\inf\{\cE(f,f)\mid f\in\cF\text{ and }f|_{V_m}=x\}.
\]
We call such $h$ a piecewise harmonic (more precisely, an $m$-harmonic) function. When $m=0$, $h$ is called a harmonic function and is denoted by $\iota(x)$.
\begin{lemma}\label{lem:dense}
For $f\in\cF$ and $m\in\Z_+$, let $f_m$ be an $m$-harmonic function such that $f_m=f$ on $V_m$.
Then, $f_m$ converges to $f$ in $\cF$ as $m\to\infty$.
In particular, the totality of piecewise harmonic functions is dense in $\cF$.
\end{lemma}
\begin{proof}
The proof is standard. 
From the maximum principle (see, e.g., \cite[Lemma 2.2.3]{Ki}), 
\[
\min_{K_w}f\le \min_{\psi_w(V_0)}f=\min_{K_w}f_m \le \max_{K_w}f_m= \max_{\psi_w(V_0)}f\le \max_{K_w}f
\]
for any $w\in \tilde W_m$. Therefore, $f_m$ converges to $f$ uniformly on $G(L)$, in particular, in $L^2(G(L),\mu_q)$ as $m\to\infty$. Because $\{f_m\}_{m\in\Z_+}$ is bounded in $\cF$, it converges to $f$ weakly in $\cF$. Because $\lim_{m\to\infty}(f_m,f_m)_\cF=(f,f)_\cF$, $f_m$ actually converges to $f$ strongly in $\cF$.
\qed
\end{proof}

Let $v\in W_*$. We define $L^{[v]}=\{L_w^{[v]}\}_{w\in W_*}\in T^{W_*}$ by $L_w^{[v]}=L_{vw}$.
Then, we can define a strongly local regular Dirichlet form $(\cE^{[v]},\cF^{[v]})$ on $L^2(G(L^{[v]}),\mu_{q}^{[v]})$, where $\mu_{q}^{[v]}$ is defined in the same way as $\mu_q$ with $L$ replaced by $L^{[v]}$. The energy measure of $f\in\cF^{[v]}$ is denoted by $\mu_{\la f\ra}^{[v]}$.
The following lemma is proved in a straightforward manner by going back to the above definition. 
\begin{lemma}\label{lem:selfsim}
\begin{enumerate}
\item Let $f\in\cF$ and $m\in\N$. For each $v\in\tilde W_m$, $f^{[v]}:=f\circ \psi_v|_{G(L^{[v]})}$ belongs to $\cF^{[v]}$. Moreover, it holds that
\begin{equation}\label{eq:selfsim1}
\cE(f,f)=\sum_{v\in \tilde W_m}\frac1{r_{v}}\cE^{[v]}(f^{[v]},f^{[v]})
\end{equation}
and
\begin{equation}\label{eq:selfsim2}
\mu_{\la f\ra}=\sum_{v\in \tilde W_m}\frac1{r_{v}}(\psi_v|_{G(L^{[v]})})_*\mu_{\la f^{[v]}\ra}^{[v]}.
\end{equation}
If $f$ is an $m$-harmonic function, then $f^{[v]}$ is a harmonic function with respect to $(\cE^{[v]},\cF^{[v]})$.
\item It holds that
\begin{equation}\label{eq:selfsim3}
\mu_q=\sum_{v\in \tilde W_m}{q_{v}}(\psi_v|_{G(L^{[v]})})_*\mu_{q}^{[v]}.
\end{equation}
\end{enumerate}
\end{lemma}

By applying \eqref{eq:selfsim1} with $\cE$ replaced by $\cE^{[\xi]}$ for $\xi\in \tilde W_*$ to $f=\iota(x)$ for $x\in l(V_0)$, we obtain the following identity as a special case: 
\begin{equation}\label{eq:Q}
 r_\xi^{-1}Q(A_{\xi} x)=\sum_{\zt\in W_m;\ \xi\zt\in\tilde W_*}r_{\xi\zt}^{-1}Q(A_{\xi\zt} x),\quad m\in\Z_+.
\end{equation}

Let $f\in\cF$.
For each $m\in\Z_+$, let $\lm_{\la f\ra}^{(m)}$ be a measure on $W_m$ defined as
\[
\lm_{\la f\ra}^{(m)}(C)=2\sum_{v\in C\cap \tilde W_m}r_v^{-1}\cE^{[v]}(f^{[v]},f^{[v]}),\quad
C\subset W_m.
\]
Then, we can verify that $\{\lm_{\la f\ra}^{(m)}\}_{m\in\Z_+}$ are consistent in the sense that $\lm_{\la f\ra}^{(m)}(C)=\lm_{\la f\ra}^{(m+1)}(C\times S)$. By the Kolmogorov extension theorem, there exists a unique Borel measure $\lm_{\la f\ra}$ on $\Sg$ such that
\[
\lm_{\la f\ra}(\Sg_C)=\lm_{\la f\ra}^{(m)}(C)
\quad \text{for any }m\in\Z_+,\ C\subset W_m,
\]
where $\Sg_C=\bigcup_{v\in C}\Sg_v$.
It is easy to see that $\lm_{\la f\ra}(\Sg\setminus \tilde\Sg)=0$.

In particular, if $f=\iota(x)$ for $x\in l(V_0)$, we have
\begin{equation}\label{eq:lm}
\lm_{\la \iota(x)\ra}(\Sg_C)=2\sum_{v\in C\cap \tilde W_m}r_v^{-1}Q(A_v x),\quad
C\subset W_m.
\end{equation}
For simplicity, we write $\lm_{\la x\ra}$ for $\lm_{\la \iota(x)\ra}$.
\begin{lemma}\label{lem:lm}
For $f\in \cF$, $(\pi|_{\tilde \Sg})_* \lm_{\la f\ra}=\mu_{\la f\ra}$.
\end{lemma}
\begin{proof}
This lemma is proved in \cite[Lemma~4.1]{Hi05} when $T$ is a one-point set. In the general case, it suffices to modify the proof line by line by using Lemma~\ref{lem:selfsim} as a substitution of the self-similar property. We provide a proof here for the reader's convenience.

We define a set function $\chi_m$ for $m\in \Z_+$ by
\[
\chi_m(A)=\sum_{v\in\tilde W_m}\frac1{r_v}\mu_{\la f^{[v]}\ra}^{[v]}(\pi(\sg_v^{-1}(A)))
\]
for a $\sg$-compact subset $A$ of $\tilde\Sg$.

Let $B$ be a closed subset of $G(L)$. For $v\in\tilde W_m$,
\begin{align*}
(\psi_v|_{G(L^{[v]})})^{-1}(B)&=\pi\bigl((\pi|_{\tilde\Sg})^{-1}((\psi_v|_{G(L^{[v]})})^{-1}(B))\bigr)\\
&=\pi\bigl(\sg_v^{-1}((\pi|_{\tilde\Sg})^{-1}(B))\bigr).
\end{align*}
Therefore, $\mu_{\la f\ra}(B)=\chi_m\bigl((\pi|_{\tilde\Sg})^{-1}(B)\bigr)$ from \eqref{eq:selfsim2}.

For $C\subset \tilde W_m$,
\begin{align*}
\lm_{\la f\ra}(\Sg_C)
&=\lm_{\la f\ra}^{(m)}(C)\\
&=2\sum_{v\in C}r_v^{-1}\cE^{[v]}(f^{[v]},f^{[v]})\\
&=\sum_{v\in \tilde W_m}r_v^{-1}\mu_{\la f^{[v]}\ra}^{[v]}\bigl(\pi(\sg_v^{-1}(\Sg_C))\bigr)\\
&=\chi_m(\Sg_C).
\end{align*}
Here, in the third equality, we used the identity
\[
\pi(\sg_v^{-1}(\Sg_C))=\begin{cases}
G(L^{[v]})& \text{if }v\in C,\\
\emptyset&\text{otherwise}.
\end{cases}
\]

Let $F$ be a closed subset of $G(L)$. Then, $(\pi|_{\tilde\Sg})^{-1}(F)$ is also closed in $\tilde\Sg$.
For $m\in\Z_+$, let $C_m=\bigl\{w\in \tilde W_m\bigm| \Sg_w\cap (\pi|_{\tilde\Sg})^{-1}(F)\ne\emptyset\bigr\}$.
Then, $\{\Sg_{C_m}\}_{m=0}^\infty$ is decreasing in $m$ and $\bigcap_{m\in\Z_+}\Sg_{C_m}=(\pi|_{\tilde\Sg})^{-1}(F)$.
By using the monotonicity of $\chi_m$,
\[
\mu_{\la f\ra}(F)=\chi_m\bigl((\pi|_{\tilde\Sg})^{-1}(F)\bigr)
\le \chi_m(\Sg_{C_m})=\lm_{\la f\ra}(\Sg_{C_m}).
\]
Letting $m\to\infty$, we have $\mu_{\la f\ra}(F)\le \lm_{\la f\ra}(F)$.

The inner regularity of $\mu_{\la f\ra}$ and $\lm_{\la f\ra}$ implies that $\mu_{\la f\ra}(B)\le \lm_{\la f\ra}(B)$ for all Borel sets $B$. Because the total measures of $\mu_{\la f\ra}$ and $\lm_{\la f\ra}$ are the same, we also have the reverse inequality by considering $G(L)\setminus B$ in place of $B$.
\qed
\end{proof}

Let $i\in S_0$ and $\nu\in T$. From \cite[Proposition A.1.1 and Theorem A.1.2]{Ki}, both $1$ and $r^{(\nu)}$ are simple eigenvalues of $A_i^{(\nu)}$, and the modulus of another eigenvalue $s^{(\nu)}$ of $A_i^{(\nu)}$ is less than $r^{(\nu)}$.
In our situation, the eigenvectors are explicitly described: the eigenvectors of eigenvalues $1$, $r^{(\nu)}$, $s^{(\nu)}$ are constant multiples of 
\begin{align*}
&\bfone:=\begin{pmatrix}1\\1\\1\end{pmatrix},\quad
\tilde v_1:=\begin{pmatrix}0\\1\\1\end{pmatrix},\quad
y_1:=\begin{pmatrix}0\\1\\-1\end{pmatrix}\quad
\text{for }A_1^{(\nu)},\\
&\bfone,\phantom{{}:=\begin{pmatrix}1\\1\\1\end{pmatrix}{}}\quad
\tilde v_2:=\begin{pmatrix}1\\0\\1\end{pmatrix},\quad
y_2:=\begin{pmatrix}-1\\0\\1\end{pmatrix}\quad
\text{for }A_2^{(\nu)},\\
&\bfone,\phantom{{}:=\begin{pmatrix}1\\1\\1\end{pmatrix}{}}\quad
\tilde v_3:=\begin{pmatrix}1\\1\\0\end{pmatrix},\quad
y_3:=\begin{pmatrix}1\\-1\\0\end{pmatrix}\quad
\text{for }A_3^{(\nu)},
\end{align*}
respectively. Here, we identify $x\in l(V_0)$ with $\begin{pmatrix}x(p_1)\\x(p_2)\\x(p_3)\end{pmatrix}$. It is crucial for subsequent arguments that the eigenvectors of eigenvalue $r^{(\nu)}$ are independent of $\nu$.

Let $\tilde l(V_0)$ be the set of all $x\in l(V_0)$ such that $\sum_{p\in V_0}x(p)=0$. The orthogonal linear space of $\tilde l(V_0)$ in $l(V_0)$ is one-dimensional and spanned by $\bfone$.
The function $\tilde l(V_0)\ni x\mapsto Q(x,x)^{1/2}\in \R$ defines a norm on $\tilde l(V_0)$.
Let $P$ denote the orthogonal projection from $l(V_0)$ onto $\tilde l(V_0)$. 
For each $i\in S_0$, $u_i\in l(V_0)$ denotes the column vector $(D_{p,p_i})_{p\in V_0}$.
\begin{lemma}[{see, e.g.,~\cite[Lemma~5]{HN06} and \cite[Lemma~A.1.4]{Ki}}] \label{lem:u}
For each $i\in S_0$ and $\nu\in T$, $u_i$ is an eigenvector of ${\,}^t\!A_i^{(\nu)}$ with respect to the eigenvalue $r^{(\nu)}$. Moreover, $u_i\in \tilde l(V_0)$.
\end{lemma}
We also note that $(u_i,\bfone)=(u_i,y_i)=0$.
We take $v_i\in l(V_0)$ such that $v_i$ is a constant multiple of $\tilde v_i$ and $(u_i,v_i)=1$.
\begin{lemma}\label{lem:A}
Let $i\in S_0$, $x\in l(V_0)$, and $\bfnu=\{\nu_k\}_{k\in\N}\in T^\N$.
Then, it holds that
\begin{equation}\label{eq:A}
\lim_{n\to\infty} r_{i^{\nu_1}i^{\nu_2}\cdots i^{\nu_n}}^{-1} P A_{i^{\nu_1}i^{\nu_2}\cdots i^{\nu_n}}x=(u_i,x) Pv_i
\end{equation}
and
\begin{equation}\label{eq:B}
\lim_{n\to\infty} r_{i^{\nu_1}i^{\nu_2}\cdots i^{\nu_n}}^{-2} Q(A_{i^{\nu_1}i^{\nu_2}\cdots i^{\nu_n}}x)=(u_i,x)^2 Q(v_i).
\end{equation}
Moreover, these convergences are uniform in $i\in S_0$, $x\in \cC$, and $\bfnu\in T^\N$, where $\cC$ is the inverse image of an arbitrary compact set of $l(V_0)$ by $P$.
\end{lemma}
\begin{proof}
Note that $P A_{i^{\nu_1}i^{\nu_2}\cdots i^{\nu_n}}\bfone=0$ and $r_{i^{\nu_1}i^{\nu_2}\cdots i^{\nu_n}}^{-1}  A_{i^{\nu_1}i^{\nu_2}\cdots i^{\nu_n}}v_i=v_i$ for all $n$.
Moreover, $|r_{i^{\nu_1}i^{\nu_2}\cdots i^{\nu_n}}^{-1}  A_{i^{\nu_1}i^{\nu_2}\cdots i^{\nu_n}}y_i|\le \theta^n|y_i|$, where $\theta=\max_{\nu\in T}|s^{(\nu)}/r^{(\nu)}|\in [0,1)$.

For  $x\in l(V_0)$ in general, we can decompose $x$ into $x=x_1\bfone+x_2 v_i+x_3 y_i$. By taking the inner product with $u_i$ on both sides, $(u_i,x)=x_2(u_i,v_i)=x_2$. Therefore, \eqref{eq:A} holds, and \eqref{eq:B} follows immediately from \eqref{eq:A}. 
The uniformity of the convergences is evident from the argument above.
\qed
\end{proof}
Although the next lemma can be confirmed by concrete calculation, we provide a proof that is applicable to more general situations.
\begin{lemma}\label{lem:C}
The following hold.
\begin{enumerate}
\item For every $i,j\in S_0$, $Q(v_i,v_i)=Q(v_j,v_j)>0$. For  $j\in S_0$ and $i,i'\in S_0\setminus\{j\}$, $(Dv_j)(p_i)=(Dv_j)(p_{i'})$.
\item For every $i,j\in S_0$, $(u_i,v_j)\ne0$.
\item There exists $\dl_0>0$ such that, for each $i\in S_0$, there exists some $i'\in S_0$ satisfying
\begin{equation}\label{eq:C}
\bigl| |(Dv_i)(p_i)|-|(Dv_i)(p_{i'})|\bigr|\ge\dl_0.
\end{equation}
\end{enumerate}
\end{lemma}
\begin{proof}
\begin{enumerate}
\item This is proved in \cite[Lemma~10]{HN06} in more-general situations.
\item Note that $(u_j,v_j)=1$. From (1), $(u_i,v_j)=(D v_j)(p_i)$ is independent of $i\in S_0\setminus\{j\}$. Moreover, $0=(D v_j,\bfone)=\sum_{i\in S_0}(u_i,v_j)$. Therefore, $(u_i,v_j)=-1/(\#S_0-1)=-1/2$ for $i\in S_0\setminus\{j\}$.
\item From the proof of (2), we can take $\dl_0=1/2$.
\qed
\end{enumerate}
\end{proof}
The following are simple estimates used in the proofs of Lemma~\ref{lem:key} and Theorem~\ref{th:main}.
\begin{lemma}\label{lem:simple}
Let $s,t>0$ and $a>0$. If $\left|\log (t/s)\right|\ge a$, then \[|t-s|\ge(1-e^{-a})\max\{s,t\}.\]
\end{lemma}
\begin{proof}
We may assume that $s\le t$. Then,
$t/s\ge e^a$, which implies $t-s\ge t-t e^{-a}=t(1-e^{-a})$.
\qed
\end{proof}
\begin{lemma}\label{lem:simple2}
Let $d\in\N$ and
\[
  \cP_d=\left\{a=(a_1,\dots,a_d)\in \R^d\;\middle|\; a_k\ge0 \text{ for all $k=1,\dots,d$, and }\sum_{k=1}^d a_k=1\right\}.
\]
For $a=(a_1,\dots,a_d),\ b=(b_1,\dots,b_d)\in \cP_d$, it holds that
\[
\sum_{k=1}^d\sqrt{a_k b_k}\le 1-\frac{|a-b|_{\R^d}^2}8.
\]
\end{lemma}
\begin{proof}
Since all $a_k$ and $b_k$ are dominated by $1$,
\begin{align*}
\sqrt{a_k b_k}&=\frac{a_k+b_k}2-\frac{(a_k-b_k)^2}{2(\sqrt{a_k}+\sqrt{b_k})^2}\\
&\le\frac{a_k+b_k}2-\frac{(a_k-b_k)^2}{8}.
\end{align*}
Taking the sum with respect to $k$ on both sides, we arrive at the conclusion.
\qed
\end{proof}

At the end of this section, we introduce a general sufficient condition for singularity of two measures. For $z\in\R$, let
\[
  z^\oplus=\begin{cases}
  1/z &(z\ne0)\\
  0& (z=0).
 \end{cases}
\]
\begin{theorem}\label{th:criterion}
Let $(\Omega,\cB,\{\cB_n\}_{n\in\Z_+})$ be a measurable space equipped with a filtration such that $\cB=\bigvee_{n\in\Z_+}\cB_n$. Let $P_1$ and $P_2$ be two probability measures on $(\Omega,\cB)$. Suppose that, for each $n\in\Z_+$, $P_2|_{\cB_n}$ is absolutely continuous with respect to $P_1|_{\cB_n}$. Let $z_n$ be the Radon--Nikodym derivative $d(P_2|_{\cB_n})/d(P_1|_{\cB_n})$ for $n\in\Z_+$ and $\a_n=z_{n}z_{n-2}^\oplus$ for $n\ge2$.
If
\begin{equation}\label{eq:criterion}
\sum_{n=2}^\infty \left(1-\E^{P_1}[\sqrt{\a_{n}}\mid \cB_{n-2}]\right)=\infty\quad P_1\text{-a.s.}
\end{equation}
holds, then $P_1$ and $P_2$ are mutually singular. Here, $\E^{P_1}[{}\mathrel{\cdot}|\cB_{n-2}]$ denotes the conditional expectation for $P_1$ given $\cB_{n-2}$.
\end{theorem}
\begin{proof}
We modify the proof of \cite[Theorem~4.1]{Hi05}.
By \cite[Theorem~VII.6.1]{Sh}, $z_\infty:=\lim_{n\to\infty}z_n$ exists $(P_1+P_2)$-a.e.\ and
\begin{equation}\label{eq:decomp}
P_2(A)=\int_A z_\infty\,d P_1+P_2(A\cap\{z_\infty=\infty\}),\quad
A\in\cB.
\end{equation}
Moreover, $P_1$ and $P_2({}\cdot{}\mathrel{\cap}\{z_\infty=\infty\})$ are mutually singular.

Let
\begin{align*}
Z_1&=\left\{\sum_{k=1}^\infty \bigl(1-\E^{P_1}[\sqrt{\a_{2k}}\mid \cB_{2(k-1)}]\bigr)=\infty\right\},\\
Z_2&=\left\{\sum_{k=1}^\infty \bigl(1-\E^{P_1}[\sqrt{\a_{2k+1}}\mid \cB_{2(k-1)+1}]\bigr)=\infty\right\}.
\end{align*}
From \eqref{eq:criterion}, $P_1(Z_1\cup Z_2)=1$.
Considering the two filtrations $\{\cB_{2k}\}_{k\in\Z_+}$ and $\{\cB_{2k+1}\}_{k\in\Z_+}$ and following the proof of \cite[Theorem~VII.6.4]{Sh}, we have $\{z_\infty=\infty\}=Z_1=Z_2$ up to $P_2$-null sets.
Therefore, $z_\infty=\infty$ $P_2$-a.e.\ on $Z_1\cup Z_2$.
Applying \eqref{eq:decomp} to $A=\Omega\setminus(Z_1\cup Z_2)$, which is a $P_1$-null set, we have $P_2(A)=P_2(A\cap\{z_\infty=\infty\})$, that is, $z_\infty=\infty$ $P_2$-a.e.\ on $A$.
Thus, $P_2(z_\infty=\infty)=1$ and we conclude that $P_1$ and $P_2$ are mutually singular.
\qed
\end{proof}
\section{Proof of the main results}
We introduce some notation.
Let $\cK$ be a closed set of $l(V_0)$ that is defined as
\[
\cK=\{x\in l(V_0)\mid 2Q(x,x)=1\}.
\]
For $l_0\in\Z_+$ and $l_1,l_2\in\N$, let
\[
L(l_0,l_1,l_2)=\left\{\bfnu=\{\nu_k\}_{k=1}^\infty\in T^\N\;\middle|\;\begin{array}{l}\{\nu_k\mid l_0+1\le k\le l_0+l_1\}\\\subset \{\nu_k\mid  l_0+l_1+1\le k\le l_0+l_1+l_2\}\end{array}\right\}.
\]
We define several constants as follows:
\begin{align*}
\b_1&:=\min\{|(u_i,v_j)|\mid i,j\in S_0\}
=\min\{|(D v_i)(p)|\mid i\in S_0,\ p\in V_0\},\\
\b_2&:=\min\left\{\left|\log(r_v/q_v)\right|\mid v\in S,\ r_v\ne q_v\right\}>0,\\
\b_3&:=\min\{q_v\mid v\in S\}>0,\\
\b_4&:=\min\{r^{(\nu)}\mid \nu\in T\}>0,\\
\b_5&:=2Q(v_i,v_i)>0 \quad(i\in S_0).
\end{align*}
By Lemma~\ref{lem:C}(2), $\b_1>0$. In the definition of $\b_2$, $\min\emptyset=1$ by convention.
By Lemma~\ref{lem:C}(1), $\b_5$ is independent of the choice of $i$.

We fix $q\in\cA$.
The following is a key lemma for proving Theorem~\ref{th:main}.
\begin{lemma}\label{lem:key}
\begin{enumerate}
\item There exist $N\in\N$ and $N'\in \N$ such that, for any $l\in\N$, there exists $\gm>0$ satisfying the following. For all $\bfnu=\{\nu_k\}_{k=1}^\infty\in L(N,N',l)$ and $x\in \cK$, there exist 
\begin{align*}
i&=i(x)\in S_0,\\
j&=j(x,\nu_1,\nu_2,\dots, \nu_N)\in S_0,\\
m&=m(l,x,\nu_1,\nu_2,\dots, \nu_{N+N'+l})\in\{N',N'+1,\dots, N'+l\}
\end{align*}
such that 
\[
|2r_\xi^{-1}Q(A_\xi x)-q_\xi|\ge\gm
\]
with $\xi=i^{\nu_1}\cdots i^{\nu_N} j^{\nu_{N+1}}\cdots j^{\nu_{N+m}}$.
Here, ``$j=j(x,\nu_1,\nu_2,\dots, \nu_N)$'' means that ``$j$ depends only on $x,\nu_1,\nu_2,\dots, \nu_N$,'' and so on.
\item If Condition~\textup{(A)} holds, then the claim of item~{\rm(1)} holds with ``$\bfnu=\{\nu_k\}_{k=1}^\infty\in L(N,N',l)$'' replaced by ``$\bfnu=\{\nu_k\}_{k=1}^\infty\in T^\N$.''
\end{enumerate}
\end{lemma}
\begin{proof}
(1) Let $\ph$ be a continuous function on $l(V_0)$ that is defined as
\[
\ph(x)=\sum_{i\in S_0}(u_i,x)^2.
\]
Since the range of $\ph$ on $\cK$ is equal to that on a compact set $P(\cK)$, $\ph$ attains a minimum on $\cK$, say $\b_6$.
Let $x\in\cK$. Because 
\[
0<Q(x,x)=(-Dx,x)=-\sum_{i\in S_0} (u_i,x)x(p_i),
\]
$(u_i,x)\ne0$ for some $i\in S_0$. This implies that $\ph(x)>0$. (In fact, we can confirm that $\ph(x)\equiv 3/2$.) Thus, $\b_6>0$.
Define $\dl'=\b_6/\#S_0=\b_6/3$ and $\cK_i=\{x\in\cK\mid (u_i,x)^2\ge\dl'\}$ for $i\in S_0$. It holds that $\cK=\bigcup_{i\in S_0}\cK_i$.

We fix $x\in \cK$. There exists $i\in S_0$ such that $x\in \cK_i$.
From Lemma~\ref{lem:C}(3), there exists $i'\in S_0$ such that \eqref{eq:C} holds. By keeping in mind that $(Dv_i)(p_i)=1$, it follows that
\begin{align}
\bigl| (Dv_i)(p_i)^2-(Dv_i)(p_{i'})^2\bigr|
&=\bigl| 1+|(Dv_i)(p_{i'})|\bigr|\,\bigl| |(Dv_i)(p_i)|-|(Dv_i)(p_{i'})|\bigr|\nonumber\\
&\ge \dl_0.\label{eq:dl_0}
\end{align}
Let $\bfnu=\{\nu_k\}_{k\in\N}\in T^\N$ and define $x_n=r_{i^{\nu_1}\cdots i^{\nu_n}}^{-1}A_{i^{\nu_1}\cdots i^{\nu_n}}x$ for $n\in\N$.
From Lemma~\ref{lem:u},
\begin{equation}\label{eq:key0}
 (u_i,x_n)=(r_{i^{\nu_1}\cdots i^{\nu_n}}^{-1}{}^t\!A_{i^{\nu_1}\cdots i^{\nu_n}} u_i,x)=(u_i,x).
\end{equation}
Let $\dl_1=\sqrt{\dl'}\b_1/2$ and $\dl_2=\dl'\dl_0/3$.
By Lemma~\ref{lem:A}, there exists $N\in\N$ independent of the choice of $x$, $i$, and $\bfnu$ such that, for all $p\in V_0$,
\begin{align}
\bigl| |(Dx_N)(p)|-|(u_i,x)(Dv_i)(p)|\bigr|&\le \dl_1\label{eq:key1}\\
\shortintertext{and}
\bigl| (Dx_N)(p)^2-(u_i,x)^2(Dv_i)(p)^2\bigr|&\le\dl_2.\label{eq:key2}
\end{align}
From \eqref{eq:key0} and \eqref{eq:key1}, for any $j\in S_0$,
\begin{align*}
|(u_j,x_N)|&=|(Dx_N)(p_j)|\\
&\ge |(u_i,x)(Dv_i)(p_j)|-\dl_1\\
&\ge \sqrt{\dl'}\b_1-\dl_1=\dl_1.
\end{align*}
By Lemma~\ref{lem:A},
\begin{align*}
  \lim_{m\to\infty}r_{j^{\nu_{N+1}}\cdots j^{\nu_{N+m}}}^{-2} Q(A_{j^{\nu_{N+1}}\cdots j^{\nu_{N+m}}}x_N)
  &=(u_j,x_N)^2 Q(v_j)\\
  &\ge \dl_1^2 \b_5/2>0.
\end{align*}
This convergence is uniform in $x$, $i$, $j$, and $\bfnu$ because $P x_N$ belongs to some compact set of $\cK$ that is independent of them. 
We take $\dl_3=\b_5\dl_2/2$.
Then, there exists $N'\in\N$ independent of $x$, $i$, $j$, and $\bfnu$ such that, for every $n\ge N'$,
\begin{equation}\label{eq:conv}
\bigl|r_{j^{\nu_{N+1}}\cdots j^{\nu_{N+n}}}^{-2} Q(A_{j^{\nu_{N+1}}\cdots j^{\nu_{N+n}}}x_N)
-(u_j,x_N)^2 Q(v_j)\bigr|\le \dl_3/4
\end{equation}
and
\begin{equation}\label{eq:log}
\left|\log\frac{r_{j^{\nu_{N+1}}\cdots j^{\nu_{N+n-1}}}^{-2}Q(A_{j^{\nu_{N+1}}\cdots j^{\nu_{N+n-1}}}x_N)}{r_{j^{\nu_{N+1}}\cdots j^{\nu_{N+n}}}^{-2}Q(A_{j^{\nu_{N+1}}\cdots j^{\nu_{N+n}}}x_N)}\right|\le \frac{\b_2}{2}.
\end{equation}
From \eqref{eq:dl_0} and \eqref{eq:key2},
\begin{align*}
\dl'\dl_0
&\le (u_i,x)^2\bigl| (Dv_i)(p_i)^2-(Dv_i)(p_{i'})^2\bigr|\\
&\le \bigl| (u_i,x)^2(Dv_i)(p_i)^2-(Dx_N)(p_i)^2\bigr|+\bigl| (Dx_N)(p_i)^2-(Dx_N)(p_{i'})^2\bigr|\\
&\quad+\bigl| (Dx_N)(p_{i'})^2-(u_i,x)^2(Dv_i)(p_{i'})^2\bigr|\\
&\le 2\dl_2+\bigl| (Dx_N)(p_i)^2-(Dx_N)(p_{i'})^2\bigr|,
\end{align*}
which implies that
\[
\bigl| (Dx_N)(p_i)^2-(Dx_N)(p_{i'})^2\bigr|\ge \dl'\dl_0-2\dl_2=\dl_2.
\]
From the identity $(Dx_N)(p_j)=(u_j,x_N)$ $(j\in S_0)$, we have
\begin{align*}
 2\dl_3&=\b_5\dl_2\\
 &\le \b_5\bigl|(u_i,x_N)^2-(u_{i'},x_N)^2\bigr|\\
 &\le \bigl|2Q(v_i)(u_i,x_N)^2-q_{w}/r_{w}\bigr|+\bigl|2Q(v_{i'})(u_{i'},x_N)^2-q_{w}/r_{w}\bigr|,
\end{align*}
where we choose $w=i^{\nu_1}\cdots i^{\nu_N}\in W_N$.
Then, for either $j=i$ or $i'$,
\begin{equation}\label{eq:iori'}
\bigl|2Q(v_j)(u_j,x_N)^2-q_{w}/r_{w}\bigr|\ge \dl_3.
\end{equation}
We fix such $j$.
Take any $l\in \N$ and suppose $\bfnu\in L(N,N',l)$.
There are two possibilities:
\begin{enumerate}
\item[I)] There exists some $k\in\{N'+1,\dots,N'+l\}$ such that $r_{j^{\nu_{N+k}}}\ne q_{j^{\nu_{N+k}}}$.
\item[II)] $r_{j^{\nu_{N+k}}}= q_{j^{\nu_{N+k}}}$ for all $k\in\{N'+1,\dots,N'+l\}$. 
\end{enumerate}

Suppose Case I). Let $w'=j^{\nu_{N+1}}\cdots j^{\nu_{N+k-1}}\in W_{k-1}$.
From \eqref{eq:log} with $n=k$,
\begin{align*}
\frac{\b_2}2
&\ge\left|\log\left(r_{j^{\nu_{N+k}}}^2\times\frac{Q(A_{j^{\nu_{N+1}}\cdots j^{\nu_{N+k-1}}}x_N)}{Q(A_{j^{\nu_{N+1}}\cdots j^{\nu_{N+k}}}x_N)}\right)\right|\\
&= \left|\log\left(\frac{r_{j^{\nu_{N+k}}}}{q_{j^{\nu_{N+k}}}}\frac{2r_{w w'}^{-1}Q(A_{w w'} x)}{q_{w w'}}\frac{q_{w w' j^{\nu_{N+k}}}}{2r_{w w'j^{\nu_{N+k}}}^{-1}Q(A_{w w' j^{\nu_{N+k}}}x)}\right)\right|\\
&\ge \b_2-\left|\log\frac{2r_{w w'}^{-1}Q(A_{w w'} x)}{q_{w w'}}\right|-\left|\log\frac{q_{w w' j^{\nu_{N+k}}}}{2r_{w w'j^{\nu_{N+k}}}^{-1}Q(A_{w w' j^{\nu_{N+k}}}x)}\right|.
\end{align*}
Therefore, either
\[
\left|\log\frac{2r_{w w'}^{-1}Q(A_{w w'} x)}{q_{w w'}}\right|\ge\frac{\b_2}{4}
\quad\text{or}\quad
\left|\log\frac{q_{w w' j^{\nu_{N+k}}}}{2r_{w w'j^{\nu_{N+k}}}^{-1}Q(A_{w w' j^{\nu_{N+k}}}x)}\right|\ge\frac{\b_2}{4}
\]
holds.
Since $q_{w w'}\ge q_{w w' j^{\nu_{N+k}}}\ge \b_3^{N+N'+l}$, Lemma~\ref{lem:simple} implies that either
\begin{equation}\label{eq:I1}
|2r_{w w'}^{-1}Q(A_{w w'} x)-q_{w w'}|\ge (1-e^{-\b_2/4})\b_3^{N+N'+l}
\end{equation}
or
\begin{equation}\label{eq:I2}
|2r_{w w'j^{\nu_{N+k}}}^{-1}Q(A_{w w' j^{\nu_{N+k}}}x)-q_{w w' j^{\nu_{N+k}}}|\ge (1-e^{-\b_2/4})\b_3^{N+N'+l}
\end{equation}
holds.

Next, suppose Case II). 
Since $\bfnu\in L(N,N',l)$, $r_{j^{\nu_{N+k}}}= q_{j^{\nu_{N+k}}}$ for all $k\in\{1,\dots,N'\}$.
Let $\hat w=j^{\nu_{N+1}}\cdots j^{\nu_{N+N'}}\in W_{N'}$.
Note that $q_{\hat w}=r_{\hat w}$.
From \eqref{eq:iori'} and \eqref{eq:conv},
\begin{align*}
\dl_3&\le|2Q(v_j)(u_j,x_N)^2-q_{w}/r_{w}|\\
&\le|2Q(v_j)(u_j,x_N)^2-2r_{\hat w}^{-2}Q(A_{\hat w} x_N)|+|2r_{\hat w}^{-2}Q(A_{\hat w} x_N)-q_{w \hat w}/r_{w \hat w}|\\
&\le \dl_3/2+\b_4^{-(N+N')}|2r_{w \hat w}^{-1}Q(A_{w \hat w} x)-q_{w \hat w}|.
\end{align*}
Therefore, 
\[
|2r_{w \hat w}^{-1}Q(A_{w \hat w} x)-q_{w \hat w}|\ge \dl_3 \b_4^{N+N'}/2.
\]

In conclusion, it suffices to take
\[
m=\begin{cases}
k-1&\text{if \eqref{eq:I1} holds in Case I),}\\
k  &\text{if \eqref{eq:I1} fails to hold in Case I),}\\
N'& \text{in Case II)}
\end{cases}
\]
and
\[
\gm=\min\bigl\{(1-e^{-\b_2/4})\b_3^{N+N'+l},\, \dl_3 \b_4^{N+N'}/2\bigr\}.
\]
(2) In the proof of (1), the condition that $\bfnu\in L(N,N',l)$ is used only in the discussion of Case II). Under Condition~(A), Case II) never happens. Therefore, the arguments are valid for all $\bfnu\in T^\N$.
\qed
\end{proof}
\begin{proof}[of Theorem~\ref{th:main}]
Let $N$ and $N'$ be natural numbers that are provided in Lemma~\ref{lem:key}.
Under Condition~(B), take $l_2\in\N$ associated with $l_0=N$ and $l_1=N'$ in (B). 
Under Condition~(A), take $l_2=1$.

Let $M=N+N'+l_2$. For $n\in\Z_+$, let $\cB_n$ denote the $\sg$-field on $\Sg$ that is generated by $\{\Sg_w\mid w\in W_{Mn}\}$.
Then, $\bigvee_{n=0}^\infty \cB_n$ is equal to the Borel $\sg$-field on $\Sg$.

Take $x\in \cK$. We first prove that $\lm_{\la x\ra}$ and $\lm_q$ are mutually singular.
For each $n\in\Z_+$, $\lm_{\la x\ra}|_{\cB_n}$ is absolutely continuous with respect to $\lm_q|_{\cB_n}$. Indeed, if $\lm_q(\Sg_w)=0$ for $w\in W_{Mn}$, then $w\notin \tilde W_{Mn}$, which implies $\lm_{\la x\ra}(\Sg_w)=0$.
Let $z_n$ denote the Radon--Nikodym derivative $d(\lm_{\la x\ra}|_{\cB_n})/d(\lm_q|_{\cB_n})$.

Under Condition~(B), take $\om=\om_1\om_2\cdots\in\tilde\Sg$ such that Condition~($\star$) is satisfied, and let $k\in\Z_+$ in ($\star$). 
Under Condition~(A), take $\om\in\tilde\Sg$ and $k\in\Z_+$ arbitrarily.

There exists a unique natural number $n\ge 2$ such that $M(n-2)\le k< M(n-1)$.
Let $w:=[\om]_{M(n-2)}\in \tilde W_{M(n-2)}$ and $\xi\in W_{2M}$. Using \eqref{eq:lm}, we have
\[
z_{n-2}=\frac{\lm_{\la x\ra}(\Sg_w)}{\lm_q(\Sg_w)}=\frac{2r_w^{-1} Q(A_w x)}{q_w} \quad\text{on $\Sg_w$}
\]
and
\[
z_{n}=\begin{cases}
\dfrac{2r_{w \xi}^{-1} Q(A_{w \xi} x)}{q_{w \xi}} & \text{if }w \xi\in \tilde W_{Mn}\\
0&\text{if }w \xi\notin \tilde W_{Mn}
\end{cases}
 \quad\text{on $\Sg_{w \xi}$}.
\]
Then, on $\Sg_{w \xi}$,
\[
\a_{n}:= z_{n} z_{n-2}^\oplus
=\begin{cases}
\dfrac{Q(A_{w \xi} x)Q(A_w x)^\oplus}{q_{\xi}r_{\xi}} & \text{if }w\xi\in \tilde W_{Mn},\\
0&\text{if }w\xi \notin \tilde W_{Mn}.
\end{cases}
\]
If $Q(A_w x)=0$, then $\a_{n}=0$ on $\Sg_w$, which implies that 
\begin{equation}\label{eq:trivial}
1-\E^{\lm_q}[\sqrt{\a_{n}}\mid\cB_{n-2}](\om)=1.
\end{equation}
Suppose that $Q(A_w x)\ne0$. Let $x'=A_w x\big/\sqrt{2Q(A_w x)}\in\cK$. Then,
\begin{align}
\E^{\lm_q}[\sqrt{\a_{n}}\mid\cB_{n-2}](\om)
&=\sum_{\xi\in W_{2M};\ w \xi\in\tilde W_{Mn}}\frac{q_{w \xi}}{q_w}\sqrt{\frac{Q(A_{w \xi} x)}{q_{\xi}r_{\xi}Q(A_w x)}}\nonumber\\
&=\sum_{\xi\in W_{2M};\ w \xi\in\tilde W_{Mn}}\sqrt{q_{\xi}\times 2r_{\xi}^{-1}Q(A_{\xi} x')}\nonumber\\
&\le 1-\frac18\sum_{\xi\in W_{2M};\ w \xi\in\tilde W_{Mn}}\bigl(q_{\xi}-2r_{\xi}^{-1}Q(A_{\xi} x')\bigr)^2.\label{eq:root}
\end{align}
Here, the last inequality follows from Lemma~\ref{lem:simple2}.

Take $\gm>0$ in Lemma~\ref{lem:key} associated with $l=l_2$.
Let 
\[
w'=\om_{M(n-2)+1}\om_{M(n-2)+2}\cdots\om_{k}\in W_{k-M(n-2)}
\quad\text{($w'=\emptyset$ if $k=M(n-2)$)}
\]
and $\gm'=\min\{\gm,\b_3^M\}$. Note that $q_{w'}\ge \b_3^M\ge \gm'$.
We consider the following two cases:
\begin{enumerate}
\item[i)] $|q_{w'}-2r_{w'}^{-1} Q(A_{w'}x')|\ge\gm\gm'/3$;
\item[ii)]$|q_{w'}-2r_{w'}^{-1} Q(A_{w'}x')|<\gm\gm'/3$.
\end{enumerate}
Suppose Case i). Letting $I=\{\zt\in W_{Mn-k}\mid w w'\zt\in\tilde W_{Mn}\}$, we have
\begin{align*}
&\smashoperator[l]{\sum_{\xi\in W_{2M};\ w \xi\in\tilde W_{Mn}}}\bigl(q_{\xi}-2r_{\xi}^{-1}Q(A_{\xi} x')\bigr)^2\\
&\ge \sum_{\zt\in I}\bigl(q_{w'\zt}-2r_{w'\zt}^{-1}Q(A_{w'\zt} x')\bigr)^2\\
&\ge \Biggl\{\sum_{\zt\in I}\bigl(q_{w'\zt}-2r_{w'\zt}^{-1}Q(A_{w'\zt} x')\bigr)\Biggr\}^2\Biggl(\sum_{\zt\in I}1\Biggr)^{-1}\\
&=\left(q_{w'}-2r_{w'}^{-1}Q(A_{w'} x')\right)^2 \left(\#I\right)^{-1} \qquad\text{(from \eqref{eq:Q})}\\
&\ge (\gm\gm'/3)^2(\# S)^{-2M}.
\end{align*}
Next, suppose Case ii). We have
\begin{align}
2r_{w'}^{-1}Q(A_{w'} x')
&> q_{w'}-\gm\gm'/3\nonumber\\
&\ge \gm'-\gm'/3= 2\gm'/3.\label{eq:caseii}
\end{align}
In particular, $Q(A_{w'} x')\ne0$.
Let $x''=A_{w'}x'\big/\sqrt{2Q(A_{w'}x')}\in \cK$.
We make several choices in order as follows:
\begin{itemize}
\item Take $i\in S_0$ associated with $x''\in\cK$ in Lemma~\ref{lem:key}.
\item Take $\nu_{k+1},\nu_{k+2},\dots, \nu_{k+N}\in T$ such that $w w' i^{\nu_{k+1}}i^{\nu_{k+2}}\cdots i^{\nu_{k+N}}\in \tilde W_{k+N}$; these are uniquely determined.
\item Take $j\in S_0$ associated with $x''\in\cK$, $i\in S_0$, and $\{\nu_{k+s}\}_{s=1}^N$ in Lemma~\ref{lem:key}.
\item Take a unique sequence $\{\nu_{s}\}_{s=k+N+1}^\infty\subset T$ such that 
\[
w w' i^{\nu_{k+1}}i^{\nu_{k+2}}\cdots i^{\nu_{k+N}}j^{\nu_{k+N+1}}j^{\nu_{k+N+2}}\cdots j^{\nu_{k+N+t}}\in W_{k+N+t}
\]
for every $t\in\N$.
\item Take $m\in\{N',N'+1,\dots, N'+l_2\}$ associated with $x''\in\cK$, $i\in S_0$, $j\in S_0$, and $\{\nu_{k+s}\}_{s=1}^\infty$ in Lemma~\ref{lem:key}.
\end{itemize}
Note that $\{\nu_{k+s}\}_{s=1}^\infty\in L(N,N',l_2)$ under Condition~(B).

Let
\[
\eta=i^{\nu_{k+1}}i^{\nu_{k+2}}\cdots i^{\nu_{k+N}}j^{\nu_{k+N+1}}j^{\nu_{k+N+2}}\cdots j^{\nu_{k+N+m}}\in W_{N+m}.
\]
Then, letting $J=\{\eta'\in W_{Mn-k-N-m}\mid w w' \eta\eta'\in\tilde W_{Mn}\}$, we have
\begin{align*}
&\smashoperator[l]{\sum_{\xi\in W_{2M};\ w \xi\in\tilde W_{Mn}}}\bigl(q_{\xi}-2r_{\xi}^{-1}Q(A_{\xi} x')\bigr)^2\\
&\ge \sum_{\eta'\in J}\bigl(q_{w'\eta\eta'}-2r_{w'\eta\eta'}^{-1}Q(A_{w'\eta\eta'} x')\bigr)^2\\
&\ge \Biggl\{\sum_{\eta'\in J}\bigl(q_{w'\eta\eta'}-2r_{w'\eta\eta'}^{-1}Q(A_{w'\eta\eta'} x')\bigr)\Biggr\}^2\Biggl(\sum_{\eta'\in J}1\Biggr)^{-1}\\
&=\bigl(q_{w'\eta}-2r_{w'\eta}^{-1}Q(A_{w'\eta} x')\bigr)^2 \left(\# J\right)^{-1}\qquad\text{(from \eqref{eq:Q})}\\
&\ge \bigl(q_{w'\eta}-2r_{w'\eta}^{-1}Q(A_{w'\eta} x')\bigr)^2(\# S)^{-2M}.
\end{align*}
Moreover,
\begin{align*}
&\bigl|q_{w'\eta}-2r_{w'\eta}^{-1}Q(A_{w'\eta} x')\bigr|\\
&=\left|q_{w'\eta}-2r_{\eta}^{-1}Q(A_{\eta} x'')\cdot 2r_{w'}^{-1}Q(A_{w'} x')\right|\\
&\ge 2r_{w'}^{-1}Q(A_{w'} x')\left|q_\eta-2r_\eta^{-1}Q(A_{\eta} x'')\right|-\left|q_{w'}-2r_{w'}^{-1}Q(A_{w'} x')\right|q_\eta\\
&\ge \frac{2\gm'}3\cdot\gm-\frac{\gm\gm'}3\cdot1
=\gm\gm'/3.
\end{align*}
Here, in the last inequality, we used \eqref{eq:caseii} and Lemma~\ref{lem:key}.

Therefore, in both Case i) and Case ii),
\begin{equation}\label{eq:ge}
\sum_{\xi\in W_{2M};\ w \xi\in\tilde W_{Mn}}\bigl(q_{\xi}-2r_{\xi}^{-1}Q(A_{\xi} x')\bigr)^2\ge (\gm\gm'/3)^2(\# S)^{-2M}.
\end{equation}
By combining \eqref{eq:root} with \eqref{eq:ge},
\begin{equation}\label{eq:keyineq}
1-\E^{\lm_q}[\sqrt{\a_{n}}\mid\cB_{n-2}](\om)\ge {(\gm\gm')^2(\# S)^{-2M}}/{72}.
\end{equation}
For $\lm_q$-a.s.\,$\om$, there are infinitely many $n$ that satisfy \eqref{eq:trivial} or \eqref{eq:keyineq}; therefore,
\[
\sum_{n=2}^\infty \left(1-\E^{\lm_q}[\sqrt{\a_{n}}\mid\cB_{n-2}]\right)=\infty\quad \lm_q\text{-a.s.}
\]
From Theorem~\ref{th:criterion}, we conclude that $\lm_{\la x\ra}\perp\lm_q$.

Take a $\sg$-compact set $B$ of $\Sg$ such that $\lm_{\la x\ra}(B)=1$ and $\lm_q(B)=0$. Recall that 
\[
V_*\setminus V_0=\left\{x\in G(L)\;\middle|\; \#(\pi|_{\tilde\Sg})^{-1}(\{x\})>1\right\}
\]
and $\mu_q(V_*\setminus V_0)=0$. Let $B'=(\pi|_{\tilde\Sg})^{-1}(V_*\setminus V_0)\cup B$.
Because $(\pi|_{\tilde\Sg})^{-1}(\pi(B'))=B'$, from Lemma~\ref{lem:lm}
\[
\mu_q(\pi(B'))=\lm_q\bigl((\pi|_{\tilde\Sg})^{-1}(\pi(B'))\bigr)=\lm_q(B')=0
\]
and
\[
\mu_{\la \iota(x)\ra}(\pi(B'))=\lm_{\la x\ra}\bigl((\pi|_{\tilde\Sg})^{-1}(\pi(B'))\bigr)=\lm_{\la x\ra}(B')\ge\lm_{\la x\ra}(B)=1.
\]
Therefore, $\mu_{\la \iota(x)\ra}\perp \mu_q$.
We have now proved that $\mu_{\la h\ra}\perp \mu_q$ for all harmonic functions $h$.

Next, let $f$ be an arbitrary $m$-piecewise harmonic function. For $v\in \tilde W_m$, we apply the above result to the Dirichlet form $(\cE^{[v]},\cF^{[v]})$ on $L^2(G(L^{[v]}),\mu_q^{[v]})$ and $f^{[v]}:=f\circ \psi_v|_{G(L^{[v]})}$ to conclude that $\mu_{\la f^{[v]}\ra}^{[v]}\perp \mu_q^{[v]}$.
Take a $\sg$-compact subset $B_v$ of $G(L^{[v]})$ such that $\mu_{\la f^{[v]}\ra}^{[v]}(G(L^{[v]})\setminus B_v)=0$ and $\mu_q^{[v]}(B_v)=0$.
Let
\[
  B=\bigcup_{v\in \tilde W_m}\psi_v(B_v)
  \quad\text{and}\quad
  \hat B=B\setminus (V_*\setminus V_0).
\]
From Lemma~\ref{lem:selfsim} and the property $\mu_q(V_*\setminus V_0)=0$, we have 
\begin{align*}
\mu_{\la f\ra}(B)
&\ge \sum_{v\in \tilde W_m}\frac1{r_v}\mu_{\la f^{[v]}\ra}^{[v]}(B_v)
=\sum_{v\in \tilde W_m}\frac2{r_v}\cE^{[v]}(f^{[v]},f^{[v]})\\
&=2\cE(f,f)=\mu_{\la f\ra}(G(L))
\end{align*}
and
\[
\mu_q(B)=\mu_q(\hat B)\le\sum_{v\in \tilde W_m}q_v \mu_q^{[v]}(B_v)=0.
\]
Therefore, $\mu_{\la f\ra}\perp \mu_q$.

For $f\in\cF$ in general, we can take a sequence $\{f_n\}_{n=1}^\infty$ of piecewise harmonic functions that converges to $f$ in $\cF$ from Lemma~\ref{lem:dense}.
For each $n\in\N$, take a Borel set $B_n$ of $G(L)$ such that $\mu_q(B_n)=0$ and $\mu_{\la f_n\ra}(G(L)\setminus B_n)=0$. Let $B=\bigcup_{n=1}^\infty B_n$.
From a general inequality
\[
\left|\sqrt{\mu_{\la g\ra}(C)}-\sqrt{\mu_{\la g'\ra}(C)}\right|\le \sqrt{\mu_{\la g-g'\ra}(C)}
\]
for $g,g'\in\cF$ and a Borel set $C$ of $G(L)$ (see, e.g., \cite[p.\,111]{FOT}), we obtain 
\[
\mu_{\la f\ra}(G(L)\setminus B)=\lim_{n\to\infty}\mu_{\la f_n\ra}(G(L)\setminus B)=0,
\]
while $\mu_q(B)=0$. Therefore, $\mu_{\la f\ra}\perp \mu_q$.
\qed
\end{proof}

Lastly, we prove Theorem~\ref{th:RSG}.
\begin{proof}[of Theorem~\ref{th:RSG}]
Since the assertion obviously holds when $\#T=1$, we may assume that $\#T\ge2$.

Let $q=\{q_v\}_{v\in S}\in \cA$.
Take $l_0,l_1\in\N$ arbitrarily and let $l_2=\# T$. 
For $\hat\om\in\hat\Omega$, $\tilde W_n(\hat\om)$ $(n\in\Z_+)$ and $\mu_q^{(\hat\om)}$ denote the set $\tilde W_n$ and the measure $\mu_q$ associated with $L(\hat\om)$, respectively. We define a probability measure $\P$ on $(\Sg\times\hat\Om,\cB(\Sg)\otimes\hat\cB)$ by
\[
  \P(A)=\int_{\hat\Om}\mu_q^{(\hat\om)}(A_{\hat\om})\,\hat P(d\hat\om),\quad A\in \cB(\Sg)\otimes\hat\cB,
\]
where $\cB(\Sg)$ denotes the Borel $\sg$-field on $\Sg$ and $A_{\hat\om}=\{\om\in\Sg\mid (\om,\hat\om)\in A\}$.
More specifically, if $A$ is expressed as $A=\Sg_w\times B$ for $w=w_1w_2\cdots w_m\in W_m$ and $B=\{\hat\om\in\hat\Om\mid L_v(\hat\om)=\tau_v \text{ for all }v\in W_{\le n}\}$ for given $m,n\in\Z_+$ with $m-1\le n$ and $\{\tau_v\}_{v\in W_{\le n}}\in T^{W_{\le n}}$, then
\begin{align*}
\P(A)&=\int_B \mu_q^{(\hat\om)}(\Sg_w)\,\hat P(d\hat\om)
=\int_B q_{w}\bfone_{\tilde W_m(\hat\om)}(w)\,\hat P(d\hat\om)\\
&=\begin{cases}
\displaystyle q_{w}\prod_{v\in W_{\le n}}\rho(\{\tau_v\})&\text{if $w_{k}\in S^{(\tau_{[w]_{k-1}})}$ for all $k=1,2,\dots,m$,}\\   
0&\text{otherwise}.
\end{cases}
\end{align*}
For $k\in\Z_+$, let $\tilde U(k)$ denote the set of all elements $(w,\hat\om)\in W_k\times\hat\Omega$ such that the following hold:
\begin{enumerate}
\item[(i)] $w\in\tilde W_k(\hat\om)$; 
\item[(ii)] for any $i,j\in S_0$, if we take $\nu_{k+1},\nu_{k+2},\dots,\nu_{k+l_0+l_1+l_2}\in T$ such that
\[
w i^{\nu_{k+1}}i^{\nu_{k+2}}\cdots i^{\nu_{k+l_0}}j^{\nu_{k+l_0+1}}j^{\nu_{k+l_0+2}}\cdots j^{\nu_{k+l_0+l_1+l_2}}\in \tilde W_{k+l_0+l_1+l_2}(\hat\om),
\]
then $\{\nu_{k+l_0+l_1+1},\nu_{k+l_0+l_1+2},\dots,\nu_{k+l_0+l_1+l_2}\}=T$.
\end{enumerate}
Define
\begin{align*}
U(k)&=\bigl\{(\om,\hat\om)\in\Sg\times\hat\Omega\bigm| ([\om]_k,\hat\om)\in\tilde U(k)\bigr\}\\
\shortintertext{and}
U_{\hat\om}(k)&=\{\om\in\Sg\mid (\om,\hat\om)\in U(k)\},\quad \hat\om\in\hat\Om.
\end{align*}
Then,
\begin{align*}
\P(U(k))
&=\int_{\hat\Om}\sum_{w\in W_k}q_w\bfone_{\tilde U(k)}(w,\hat\om)\,\hat P(d\hat\om)\\
&=\sum_{w\in W_k}q_w \hat P\bigl(\{\hat\om\in\hat\Om\mid w\in \tilde W_k(\hat\om)\}\bigr)\Biggl(l_2!\prod_{\nu\in T}\rho(\{\nu\})\Biggr)^{\#(S_0\times S_0)}\\
&=p\sum_{\nu_1,\dots,\nu_k\in T}\sum_{\substack{w_j\in S^{(\nu_j)};\\j=1,\dots,k}}\prod_{m=1}^k q_{w_m}\prod_{m=1}^{k}\rho(\{\nu_m\})\\*
&\qquad\textstyle\bigl(p:=\bigl(l_2!\prod_{\nu\in T}\rho(\{\nu\})\bigr)^9\in(0,1)\bigr)\\
&=p\Biggl(\sum_{\nu\in T}\sum_{v\in S^{(\nu)}}q_{v}\rho(\{\nu\})\Biggr)^k\\
&=p.
\end{align*}
In a similar way, we can confirm that $\{U((l_0+l_1+l_2)n)\}_{n\in\Z_+}$ are independent with respect to $\P$. 

For $0\le M<N$, we define
\begin{align*}
F_{M,N}&=\bigcap_{n=M+1}^N\bigl((\Sg\times\hat\Om)\setminus U((l_0+l_1+l_2)n)\bigr),\\
F_{M,N,\hat\om}&=\{\om\in \Sg\mid (\om,\hat\om)\in F_{M,N}\},\quad
F_{M,\hat\om}=\bigcap_{N=M+1}^\infty F_{M,N,\hat\om}\quad (\hat\om\in\hat\Om),\\
G_{M,N}&=\bigl\{\hat \om\in\hat\Om\bigm| \mu_q^{(\hat\om)}(F_{M,N,\hat\om})\ge(1-p)^{N/2}\bigr\},
\quad G_M=\limsup_{N\to\infty}G_{M,N}.
\end{align*}
Then,
\begin{align*}
\hat P(G_{M,N})
&\le (1-p)^{-N/2}\int_{\hat\Om}\mu_q^{(\hat\om)}(F_{M,N,\hat\om})\,\hat P(d\hat\om)\\
&= (1-p)^{-N/2}\P(F_{M,N})\\
&=(1-p)^{-N/2}(1-p)^{N-M}\\
&=(1-p)^{(N/2)-M}.
\end{align*}
From the Borel--Cantelli lemma, $\hat P(G_M)=0$.
Let 
\[
\cU_q=\bigl\{q'=\{q'_v\}_{v\in S}\in\cA\bigm| q'_v/q_v<(1-p)^{-1/(4(l_0+l_1+l_2))}\text{ for all }v\in S\bigr\},
\]
which is an open neighborhood of $q$ in $\cA$.
By letting $\cF_n=\sg(\{\Sg_w\mid w\in W_{(l_0+l_1+l_2)n}\})$ for $n\in\Z_+$, we have 
\[
\frac{d\bigl(\mu_{q'}^{(\hat\om)}|_{\cF_n}\bigr)}{d\bigl(\mu_{q}^{(\hat\om)}|_{\cF_n}\bigr)}\le(1-p)^{-n/4}
\quad\mu_{q}^{(\hat\om)}\text{-a.e.}
\]
for all $q'\in \cU_q$ and $\hat\om\in\hat\Om$.
Suppose that $q'\in \cU_q$ and $\hat\om\notin G_M$.
For sufficiently large $N\in\N$, $\hat\om\notin G_{M,N}$.
Because $F_{M,N,\hat\om}$ belongs to $\cF_N$, we have
\[
\mu_{q'}^{(\hat\om)}(F_{M,N,\hat\om})
\le (1-p)^{-N/4}\mu_{q}^{(\hat\om)}(F_{M,N,\hat\om})
\le(1-p)^{N/4}
\]
for large $N$, which implies $\mu_{q'}^{(\hat\om)}(F_{M,\hat\om})=0$.
Let $G(q)$ denote $\bigcup_{M\in\Z_+}G_M$. (Here, we specify the dependency of $q$.)
This is a $\hat P$-null set. 
If $\hat\om\notin G(q)$, then\linebreak $\mu_{q'}^{(\hat\om)}\bigl(\bigcup_{M\in\Z_+}F_{M,\hat\om}\bigr)=0$ for $q'\in\cU_q$, which means that 
\[
\mu_{q'}^{(\hat\om)}\Bigl(\limsup_{n\to\infty}U_{\hat\om}((l_0+l_1+l_2)n)\Bigr)=1, \quad q'\in\cU_q.
\]
Because $\cA$ is $\sg$-compact, we can take a countable subset $\{q_\a\mid\a\in\N\}$ of $\cA$ such that $\bigcup_{\a\in\N}\cU_{q_\a}=\cA$.
Let $\cN=\bigcup_{\a\in\N}G(q_\a)$.
Then, $\hat P(\cN)=0$ and for $\hat\om\in\hat\Om\setminus\cN$, 
\[
  \mu_q^{(\hat\om)}\Bigl(\limsup_{k\to\infty}U_{\hat\om}(k)\Bigr)=1,
  \quad q\in\cA.
\]
This implies that, for $\hat\om\in\hat\Om\setminus\cN$, $(\star)$ holds with \eqref{eq:A2} replaced by \eqref{eq:A2'} for $l_2=\#T$.
\qed
\end{proof}
\section{Concluding remarks}
We make some remarks about the main results.
\begin{enumerate}
\item The arguments in this paper are valid for some other inhomogeneous fractals. For example, we can obtain similar results for higher-dimensional inhomogeneous Sierpinski gaskets. A crucial property required here is that the eigenfunctions of $A_i^{(\nu)}$ ($i\in S_0$) associated with the eigenvalues $r^{(\nu)}$ do not depend on $\nu$.
\item Since Condition~(B) in Theorem~\ref{th:main} is a rather technical constraint, we focus on arguments that are valid more generally and we do not try to make the assumption as weak as possible by relying on concrete structures of fractals under consideration. Indeed, in Lemma~\ref{lem:C}(3), the part ``there exists some $i'\in S_0$'' can be strengthened to ``any $i'\in S_0\setminus\{i\}$.'' As a result, in Condition~($\star$), the part ``for every $i,j\in S_0$'' can be weakened to ``for every $i\in S_0$, for $j=i$ and for some other $j\in S_0$.''
\item We reason that Theorem~\ref{th:main} holds true without assuming Condition~(A) or (B) in practice.
\end{enumerate}

\end{document}